\documentclass{amsart}
\usepackage{epsfig,mathtools,amsmath,amssymb,amscd,graphicx,color,bbm,tikz}
\usetikzlibrary{matrix,calc}
\newtheorem{thm}{Theorem}[section]
\newtheorem{lem}[thm]{Lemma}
\newcommand{\td}{Tomasz Downarowicz}
\newcommand{\D}{Downarowicz}
\newcommand{\pw}{Faculty of Mathematics and Faculty of Fundamental Problems of Technology, Wroc\l aw University of Technology, Wybrze\.ze Wyspia\'nskiego~27, 50-370 Wroc\l aw, Poland}
\newcommand{\imp}{Institute of Mathematics, Polish Academy of Sciences, ul. \'Sniadeckich~8, 00-656 Warszawa, Poland}
\newcommand{\guang}{School of Mathematical Sciences, South China Normal University, Guangzhou, Guangdong 510631, China}
\newcommand{\upad}{Institut f\"ur Mathematik, Universit\"at Paderborn, Warburger\break Strasse~100, 33098 Paderborn, Germany}
\newcommand{\pwr}{@pwr.edu.pl}
\newcommand{\ustc}{@mail.ustc.edu.cn}
\newcommand{\padeborn}{@math.uni-paderborn.de}
\newcommand{\mpt}{measure-preserving system}
\newcommand{\im}{invariant measure}

\newcommand{\z}{\mathbb Z}
\newcommand{\C}{\mathfrak C}
\newcommand{\na}{\mathbb N}
\newcommand{\p}{\mathcal P}
\newcommand{\U}{\mathcal U}
\newcommand{\R}{\mathcal R}
\newcommand{\tl}{topological}
\newcommand{\ds}{dynamical system}
\newcommand{\usc}{upper semicontinuous}
\newcommand{\xt}{$(X,T)$}
\newcommand{\xmt}{$(X,\mu,T)$}
\newcommand{\mtx}{\mathcal M_T(X)}
\newcommand{\xtd}{$(X'',T'')$}
\newcommand{\xmtd}{$(X'',\mu'',T'')$}
\newcommand{\mtxd}{\mathcal M_{T''}(X'')}
\newcommand{\msy}{\mathcal M_S(Y)}
\newcommand{\mtxe}{\mathcal M^e_T(X)}
\newcommand{\ex}{\mathsf{ex}}
\newcommand{\diam}{\mathsf{diam}}
\newcommand{\dist}{\mathsf{dist}}
\newcommand{\id}{\mathsf{id}}
\newcommand{\Ret}{\mathsf{Ret}}
\newcommand{\sq}{sequence}
\newcommand{\zd}{zero-dimensional}
\numberwithin{equation}{section}

\begin{document}

\title{Zero-dimensional isomorphic dynamical models}

\author[T. \D]{\td}\address{\td: \pw}\email{Tomasz.\D\pwr}
\author[L. Jin]{Lei Jin}\address{Lei Jin: \imp}\email{jinleim\ustc}
\author[W. Lusky]{Wolfgang Lusky}\address{Wolfgang Lusky: \upad}\email{lusky\padeborn}
\author[Y. Qiao]{Yixiao Qiao}\address{Yixiao Qiao: \guang}\email{yxqiao\ustc}

\thanks{The authors are partially supported by the National Science Center (Poland) grant 2013/08/A/ST1/00275.}


\begin{abstract}
By an \emph{assignment} we mean a mapping from a Choquet simplex $K$ to probability \mpt s, obeying some natural restrictions. We prove that if $\Phi$ is an aperiodic assignment on a Choquet simplex $K$ such that the set of extreme points $\ex K$ is a countable union $\bigcup_n E_n$, where each set $E_n$ is compact, zero-dimensional, and the restriction of $\Phi$ to the Bauer simplex $K_n$ spanned by $E_n$ can be ``embedded'' in some \tl\ \ds, then $\Phi$ can be ``realized'' in a \zd\ system. 
\end{abstract}

\maketitle

\section{Introduction}
For a convex set $K$ we shall denote by $\ex K$ the set of its extreme points. A \emph{metrizable Choquet simplex} (which in this note will be briefly called just a \emph{simplex}) is a compact convex subset $K$ of a metric-linear space endowed with a convex metric, such that every point $p\in K$ is the barycenter of a unique probability distribution supported by $\ex K$ (we are using the Choquet-Meyer characterization of simplices, see e.g., \cite{Phelps}). A simplex $K$ is called \emph{Bauer} if $\ex K$ is closed (and hence compact). Any simplex is a compact connected space, but its set of extreme points may have various \tl\ attributes: it may but need not be compact, sigma-compact or \zd, etc. We will shortly say that the simplex has any of the above \tl\ properties meaning that its set of extreme points has it. For instance, we will speak about \zd, sigma-compact simplices. Exception: since Bauer simplices already have their name, we will not use the confusing term ``compact simplex''.

A \emph{face} of a simplex $K$ is a compact convex subset $F\subset K$ such that $\ex F\subset\ex K$. A face of a simplex is a simplex. 

By a \emph{\tl\ \ds} we mean a pair \xt, where $X$ is a compact metric space and $T:X\to X$ is a homeomorphism. It is well known that the set $\mtx$ of all $T$-invariant Borel probability measures (in the sequel we will skip the adjectives ``Borel'' and ``probability'') on $X$, equipped with the weak-star topology, is a simplex, and we shall call it \emph{the simplex of \im s}. The standard metric on measures
$$
\dist(\mu,\nu)=\sum_{n\ge 1}2^{-n}\left|\int f_n\,d\mu - \int f_n\,d\nu\right|,
$$
where $(f_n)_{n\ge 1}$ is some fixed \sq\ of normed continuous functions linearly dense in $C(X)$ (the space of continuous real functions on $X$ with the uniform norm), is well known to be convex and it induces the weak-star topology. We denote by $\mtxe$ the collection of all ergodic $T$-\im s on $X$, which coincides with the collection of extreme points of $\mtx$. By a \emph{\zd\ system} we understand a \tl\ \ds\ \xt, where the space $X$ is \zd.

By an \emph{assignment} we will understand a mapping $\Phi$ defined on a simplex $K$, whose ``values'' are \mpt s (by a \mpt\ we mean a standard probability space $(X,\Sigma,\mu)$ together with a measure-automorphism $T:X\to X$). We also require that the assignment obeys the following rules:
\begin{enumerate}
\item[(1)] extreme points are assigned ergodic \mpt s,
\item[(2)] if $p\in K$ is the barycenter of a probability distribution $\xi$ on $\ex K$ then
$$
\Phi(p)\approx\int\Phi(e)\,d\xi(e),
$$
where ``$\approx$'' denotes measure-theoretic isomorphism and the integral is realized on a disjoint union of the spaces realizing the assignments $\Phi(e)$ for $e\in\ex K$ (we will say that $\Phi$ is \emph{harmonic}).
\end{enumerate}
Two assignments, $\Phi$ and $\Phi'$ defined on $K$ and $K'$, respectively, are said to be \emph{equivalent} if there exists an affine homeomorphism $\pi:K\to K'$ such that $\Phi(p)\approx\Phi'(\pi(p))$, for every $p\in K$. Since $\Phi$ and $\Phi'$ are both harmonic, it suffices to check $\Phi(e)\approx\Phi'(\pi(e))$ for $e\in\ex K$. By a \emph{face of an assignment} $\Phi$ on a simplex $K$ we mean the restriction of $\Phi$ to a face of $K$. 

An assignment $\Phi$ on $K$ is called \emph{aperiodic} if $\Phi(e)$ is aperiodic for each $e \in \ex K$. 

For a \tl\ \ds\ \xt, the assignment $\Phi$ on $\mtx$ defined by $\Phi(\mu)=(X,\mathsf{Borel}(X),\mu,T)$ is called the \emph{natural} assignment of \xt. We say that an assignment can be \emph{realized} (\emph{embedded}) in a \tl\ \ds\ \xt\ if it is equivalent to (a face of) the natural assignment of \xt. 

\medskip
The general question about a characterization of assignments realizable in \tl\ \ds s is wide open and seems to be hopelessly difficult. It is not only the question about a possible affine-topological shape of the set $\mtx$ in a \tl\ \ds\ \xt\ but also about the possible configuration of ergodic systems placed over the extreme points of this set. For example, one can ask whether it is possible to have, in one \tl\ \ds\ (perhaps minimal), a \sq\ of measures isomorphic to, say, irrational rotations, converging to, say, a Bernoulli measure. Or is it possible to have a closed arc of Bernoulli measures parametrized continuously and increasingly by their entropies, and no other ergodic measures. The variety of imaginable questions of this kind is endless. So far, there exist only partial results, and most of them concern \zd\ systems. Let us review briefly some of them.
\begin{enumerate}
	\item On any simplex $K$ there exists an assignment realizable in a minimal \zd\ system 
	(more precisely, in a Toeplitz subshift, \cite{D91}).
	\item On any simplex $K$, given any nonnegative affine function $h$ on $K$, of the class LU (increasing 
	limit of a \sq\ of upper semicontinuous functions), there exists an assignment realizable in a 
	\zd\ minimal system, such that the resulting entropy function on \im s coincides with $h$. If $h$ is 
	upper semicontinuous, the minimal system can be a subshift (\cite{DS}).
	\item The natural assignment of any aperiodic \zd\ system can be realized in a minimal \zd\ system 	 
	(\cite{D06}).
	\item Any aperiodic assignment that can be embedded in a \zd\ system can also be realized in a \zd\ system (\cite{DFace}).
	\item If $K$ is a simplex such that $\ex K$ is countable, then any aperiodic assignment on $K$ can be realized in a Cantor minimal system (\cite{KO}).
\end{enumerate}
The result (5) allows to answer positively all questions of the kind ``can there be a \sq\ of such and such measures converging to such and such measure''. The result (4) has many applications, in particular, it allows to settle the above mentioned question about an arc of Bernoulli measures, it also allows to construct universal \tl\ systems (even minimal) which contain (up to isomorphism) every possible aperiodic ergodic system (both invertible and noninvertible). The result (4) will be heavily used also in this note.
\medskip

Since all above results concern \zd\ systems, the following question seems to be of crucial importance toward understanding the assignments realizable in all \tl\ \ds s:
\begin{itemize}
	\item[(*)] Is every natural assignment arising from an aperiodic \tl\ \ds\ equivalent to a natural assignment arising from a \zd\ \ds?
\end{itemize}
Aperiodicity restriction is added in order to avoid some trivial counterexamples, such as the identity map on a connected space. There is a large class of systems which have so-called \emph{small boundary property} (SBP), for instance all invertible finite entropy systems which possess an aperiodic minimal factor (see \cite{Lindenstrauss}). It is fairly obvious that systems with SBP have so-called \emph{isomorphic} \zd\ extensions (see e.g. \cite{BD}), and since isomorphic extensions preserve natural assignments (up to equivalence), for such systems the answer to the question (*) is positive. It is worth mentioning that it remains an open problem whether possessing an isomorphic \zd\ extension is equivalent to SBP. But the property asked for in (*) is much weaker than possessing an isomorphic \zd\ extension, and the answer to that question beyond systems with SBP is unknown and seems hopelessly difficult. Any progress in this direction is valuable. In this paper we will provide the following partial answer:

\begin{thm}\label{bauer}
Let \xt\ be an aperiodic \tl\ \ds\ such that $\mtx$ is a \zd\ Bauer simplex, or a \zd\ sigma-compact simplex.\footnote{Recall that we consider properties of $\mtxe$. We remark that a set is \zd\ and sigma-compact if and only if it is a countable union of  \zd\ compact sets.} Then the natural assignment arising from the system \xt\ can be realized in a \zd\ system.
\end{thm}

In fact, our main achievement is the following, slightly more general, theorem (of which Theorem \ref{bauer} is an obvious particular case):

\begin{thm}\label{sigma}
Suppose that $\Phi$ is an aperiodic assignment on a simplex $K$ such that $\ex K$ is a countable union $\bigcup_{n\ge 1}E_n$, where every set $E_n$ is \zd\ and compact. Assume that for each $n$ the restriction of $\Phi$ to the Bauer simplex $K_n$ spanned by $E_n$ can be embedded in a \tl\ \ds. Then $\Phi$ can be realized in a \zd\ system.
\end{thm}

The difference between the above two theorems is that in the latter we demand \emph{separate} embeddings for the restrictions $\Phi|_{K_n}$. This does not directly imply the existence of a joined realization for $\Phi$, which is assumed in the former theorem (obviously, the converse implication holds).

\medskip
\section{Preliminaries}
In this section, we summarise necessary notions, in particular, we discuss array systems and markers in aperiodic systems. For details we refer to \cite{Boy,DK}. We let \xt\ and $(Y,S)$ be \tl\ \ds s.

For $\mu\in\mtx$, a point $x\in X$ is said to be \emph{generic} for $\mu$ if the \sq\ of measures $(1/n)\sum_{i=0}^{n-1}\delta_{T^i(x)}$ tends to the measure $\mu$ as $n\to+\infty$, in the weak-star topology, where $\delta_x$ denotes the Dirac measure at the point $x$. It is well known that for any $\mu\in\mtxe$, the set of generic points in $X$ has full $\mu$-measure. We say that $\msy$ is a \emph{copy} of $\mtx$ if $(Y,S)$ and \xt\ are \tl ly conjugate. It is clear that then the natural assignments on $\msy$ and $\mtx$ are equivalent via a mapping $\pi$ induced by the \tl\ conjugacy.

Let $\Lambda_1,\Lambda_2,\dots$ be finite sets each containing at least two elements (called \emph{alphabets}, endowed with the discrete topology, the cardinalities need not be bounded). By an \emph{array system} (over $\Lambda_1,\Lambda_2,\dots$) we mean any closed, shift-invariant subset of the Cartesian product $\prod_{k\ge 1} \Lambda_k^\z$ (endowed with the product topology). Each element of the array system can be pictured as an array $x=(x_{k,n})_{k\ge1,n\in\z}$, such that each \emph{symbol} $x_{k,n}$ belongs to $\Lambda_k$. Any finite array of the form 
$$
a=a_{[1,k]\times[0,n-1]}=(a_{i,j})_{1\le i\le k,\,0\le j\le n-1} 
$$
with each entry $a_{i,j}$ belonging to $\Lambda_i$, will be called a \emph{$(k\times n)$-rectangle}.
The product topology is generated by the collection of all cylinder sets corresponding to centered $(k\times(2n\!\!+\!\!1))$-rectangles $a$ (with $k$ and $n$ ranging over $\na$), defined as follows
$$
[a] = \{x:x_{[1,k]\times[-n,n]}=a_{[1,k]\times[0,2n]}\}.
$$ 

The array system is by default regarded with the action of the horizontal shift $\sigma$ given by
$$
(\sigma(x))_{k,n}=x_{k,n+1},\ \ \ x=(x_{k,n})_{k\ge1,n\in\z}.
$$
Speaking about an array we will refer to the indices $k\ge1$ and $n\in\z$ as \emph{vertical} and \emph{horizontal} coordinates (positions), respectively.

Notice that regardless of the cardinalities of alphabets $\Lambda_k$ (as long as each of these cardinalities is at least 2), the product $\prod_k\Lambda_k$ is homeomorphic to the Cantor set $\C$. Thus the system $\prod_k\Lambda_k^\z$ (with the horizontal shift $\sigma$),
is conjugate to $\C^\z$ (with the shift). The latter is the \emph{universal zero-dimensional system} in the sense that any \zd\ system is \tl ly conjugate to a subsystem of $(\C^\z,\sigma)$. We shall call the simplex of \im s of the universal system $(\C^\z,\sigma)$ the \emph{universal simplex}. Recall that it has the affine-topological structure of the \emph{Poulsen simplex}, i.e., its extreme points form a dense subset. 

Let us return to the case of a general \tl\ \ds\ \xt. Let $(\p^{(k)})_{k\in\na}$ be a \sq\ of finite Borel-measurable partitions of $X$. For each $k$, let $\lambda\mapsto P_\lambda$ be a bijection from a finite alphabet $\Lambda_k$ onto $\p^{(k)}$. By the \emph{array-name} of a point $x\in X$ under the action of $T$ with respect to $(\p^{(k)})_{k\in\na}$ we shall mean the array $(x_{k,n})_{k\ge1,n\in\z}$ obtained by the rule:
\begin{enumerate}
\item[$\bullet$] the value $x_{k,n}$ is the symbol $\lambda\in\Lambda_k$ if and only if the point $T^n(x)$ is in the subset $P_\lambda\in\p^{(k)}$ of $X$.
\end{enumerate}
Notice that the closure of all array-names (with respect to $(\p^{(k)})_{k\in\na}$) is an array system. Note that if $X$ is \zd\ then there exists a \sq\ of clopen partitions $(\p^{(k)})_k$ which separate points. It is not hard to see that then \xt\ is \tl ly conjugate to the array system obtained as the collection of all corresponding array-names (which in this case is already closed).

Let \xt\ be a \zd\ system. By an \emph{$n$-marker} we mean a \emph{clopen} set $F\subset X$ such that
\begin{enumerate}
\item[(1)] no orbit visits $F$ twice in $n$ steps  (i.e., $F,T^{-1}F,\dots,T^{-(n-1)}F$ are disjoint),
\item[(2)] every orbit visits $F$ at least once (by compactness, this implies that for some $N\in\na$, we have $F\cup T^{-1}F\cup\dots\cup T^{-(N-1)}F=X$).
\end{enumerate}
We have the following key fact:
\begin{thm}[Krieger's Marker Lemma, aperiodic case]\label{krieger}
If \xt\ is an aperiodic \zd\ system then for every $n\in\na$ there exists an $n$-marker. The parameter $N$ in \emph{(2)} above can be selected equal to $2n-1$.
\end{thm}

\medskip
\section{Special case}
To prove Theorem \ref{sigma}, we first deal with the following special case which will be used in the main proof.

\begin{thm}\label{main}
Assume that $\Phi$ is an aperiodic assignment on a \zd\ Bauer simplex $K$ and that $\Phi$ can be embedded in a \tl\ \ds. Then $\Phi$ can be realized in a \zd\ system.
\end{thm}

Theorem \ref{main} will be proved in two major steps. At first, in Theorem \ref{submain} below, we give a slightly weaker statement, saying that the assignment $\Phi$ can be \emph{embedded} in a \zd\ system. We will take care of surjectivity later.

\begin{thm}\label{submain}
Assume that $\Phi$ is an aperiodic assignment on a \zd\ Bauer simplex $K$ and that $\Phi$ can be embedded in a \tl\ \ds. Then $\Phi$ can be embedded in a \zd\ system.
\end{thm}

\begin{proof}
Let \xt\ be a \tl\ \ds. Assume that $K$ is a \zd\ Bauer simplex which is a face in $\mtx$, and that $K$ contains no periodic measures. We need to construct a \zd\ system \xtd\ and a face $K''$ of $\mtxd$ such that the assignments on $K$ and $K''$ obtained as the restrictions of the natural assignments arising from \xt\ and \xtd, respectively, are equivalent.

We start by inductively constructing a one-parameter family $(\p^{(k)}_t)_{k\in\na}$ with $t\in[0,1]$ of \sq s of partitions of $X$, such that
$$
\lim_k\diam(\p^{(k)}_t)=0,
$$
uniformly in $t$.

Fix a decreasing to zero \sq\ $(r_1^{(k)})_{k\in\na}$. For each $k$ choose a finite open cover $\U_1^{(k)}$ of $X$ which consists of open balls $B^{(k)}(x_1^{(k)},r_1^{(k)}),\dots,B^{(k)}(x_{m_k}^{(k)},r_1^{(k)})$ of radius $r_1^{(k)}$. There exists a positive number $r_0^{(k)}<r_1^{(k)}$ such that the balls $B^{(k)}(x_1^{(k)},r_0^{(k)}),\dots,B^{(k)}(x_{m_k}^{(k)},r_0^{(k)})$ still form a cover $\U_0^{(k)}$ of $X$. The covers $\U_t^{(k)}$ for $t\in[0,1]$ are then constituted by the balls $B^{(k)}(x_1^{(k)},r_t^{(k)}),\dots,B^{(k)}(x_{m_k}^{(k)},r_t^{(k)})$ with radii $r_t^{(k)}\in[r_0^{(k)},r_1^{(k)}]$ depending continuously and increasingly on $t$.

Next, for each $k\in\na$ and $t\in[0,1]$, we let $\p^{(k)}_t$ be the finite partition of $X$ consisting of all intersections of the form
$$
P_t^\eta=\bigcap_{i=1}^{m_k}(B_i)^{\eta(i)},
$$
where $\eta:\{1,\dots,m_k\}\to\{0,1\}$, $\U_t^{(k)}=\{B_1,\dots,B_{m_k}\}$, $(B_i)^0=B_i$, $(B_i)^1=X\setminus B_i$ (some of the sets $P_t^\eta$ may be empty). Since for $\eta\equiv 1$ the set $P_t^\eta$ is empty (and we agree that the empty set has diameter zero), all members $P_t^\eta$ of $\p^{(k)}_t$ have diameters at most $2r_1^{(k)}$. Thus, $\lim_k\diam(\p^{(k)}_t)=0$, uniformly in $t$, as required.
\medskip

For a family $\p$ of subsets of $X$ we denote by $\partial\p$ the union of all boundaries of the members of $\p$. Notice that $\partial\p^{(k)}_t=\partial\U_t^{(k)}$, which equals the union of boundaries of finitely many balls of radius $r_t^{(k)}$ and with centers not depending on $t$. Since the boundary of a ball of radius $r$ is contained in the sphere of radius $r$, the boundaries of balls with a common center and different radii are pairwise disjoint (in a zero-dimensional space two balls with a common center and of different radii may be equal, but then their boundary is necessarily empty). Thus at most countably many of the boundaries of balls with a common center may have positive measure for a fixed probability measure $\mu$. It follows that the set
$$
I_\mu^+=\Bigl\{t\in[0,1]:\mu\bigl(\bigcup_{k\in\na}\partial\p_t^{(k)}\bigr)>0\Bigr\}
$$
is at most countable. Thus, for any $\mu\in\mtxe$, the set
$$
I_\mu^0=\C\setminus I_\mu^+
$$
is a dense subset of the classical Cantor set $\C\subset[0,1]$.
\medskip

We suspend the main proof for a while and prove two auxiliary lemmas.
The first one is about the following upper semicontinuity:

\begin{lem}\label{usc}
For every $k\in\na$ the function
$$
\psi^{(k)}:\mtxe\times[0,1]\to[0,1],\;\,\;(\mu,t)\mapsto\mu(\partial\p_t^{(k)})
$$
is \usc\ (of two variables).
\end{lem}

\begin{proof}
Fix $k\in\na$. We have
$$
\partial\p^{(k)}_t=\bigcup_{i=1}^{m_k}\partial B^{(k)}(x_i^{(k)},r_t^{(k)}).
$$
Hence,
$$
\mathbbm 1_{\partial\p^{(k)}_t}= \max_{1\le i\le m_k}\mathbbm 1_{\partial B^{(k)}(x_i^{(k)},r_t^{(k)})}.
$$

Given $t\in[0,1]$ and $d>0$, we denote by $h^{(k)}_{t,d}:\mathbb R\to[0,1]$ the continuous tent function assuming the value $1$ at $r_t^{(k)}$ and $0$ outside the interval $[r_t^{(k)}-d,r_t^{(k)}+d]$. Now, for $1\le i\le m_k$, we define, for all $x\in X$,
$$
g_{t,d,i}^{(k)}(x) = h^{(k)}_{t,d}(\dist(x,x_i^{(k)})) \text{ \ and \ } g_{t,d}^{(k)}=\max_{1\le i\le m_k}g_{t,d,i}^{(k)}.
$$
Clearly, for each $t$ and $d$, $g_{t,d}^{(k)}$ is a continuous function, moreover, for fixed $d>0$, the family of functions $\{g_{t,d}^{(k)}:t\in[0,1]\}$ is equicontinuous, which easily implies the double continuity of
$$
(\mu,t)\mapsto\int g_{t,d}^{(k)}\,d\mu
$$
on $\mtxe\times[0,1]$.

Further, as $d\to0$, $g_{t,d}^{(k)}$ tends non-increasingly to $\mathbbm 1_{\partial\p^{(k)}_t}$, hence, by the Dominated Lebesgue Theorem, $\int g_{t,d}^{(k)}\,d\mu$ tends non-increasingly to $\mu(\partial\p_t^{(k)})=\psi^{(k)}(\mu,t)$. Since a non-increasing limit of continuous functions is \usc, we have completed the proof of the lemma.
\end{proof}

Our next goal is to prove the following ``continuous selector lemma'':
\begin{lem}\label{selector}
Let $K$ be as in the formulation of Theorem \ref{submain}. Then there exists a continuous function $s:\ex K\to\C$ satisfying $s(\mu)\in I_\mu^0$, for every $\mu\in\ex K$.
\end{lem}

\begin{proof} We can assume that $\diam(\ex K)=1$.
For every natural number $n$, we will inductively define a finite clopen partition $\{K^{(n)}_1,\cdots,K^{(n)}_{l_n}\}$ of $\ex K$ and a finite family $\{W^{(n)}_1,\cdots,W^{(n)}_{l_n}\}$ of clopen subsets of $\C$ having, for each $n\in\na$, the following properties:
\begin{enumerate}
\item [(1)] $\max\{\diam(K^{(n)}_i),\diam(W^{(n)}_i)\}\le2^{1-n}$, for every $1\le i\le l_n$,
\item [(2)] if $n>1$ then for each $1\le i\le l_n$ there is some $1\le j\le l_{n-1}$ such that
$K^{(n)}_i\times W^{(n)}_i$ is a subset of $K^{(n-1)}_j\times W^{(n-1)}_j$,
\item [(3)] $\psi^{(k)}(K^{(n)}_i\times W^{(n)}_i)\subset[0,2^{1-n}]$, for every $0\le k\le n$ and
$1\le i\le l_n$.
\end{enumerate}

\medskip

To begin with, we set $l_1=1$, $K^{(1)}_1=\ex K$ and $W^{(1)}_1=\C$. Clearly, the conditions (1) and (3) are fulfilled. Next, we fix an $n\ge 2$ and suppose that $\{K^{(n-1)}_1,\cdots,K^{(n-1)}_{l_{n-1}}\}$ and $\{W^{(n-1)}_1,\cdots,W^{(n-1)}_{l_{n-1}}\}$ have been defined. For each $\mu\in\ex K$ there is a unique $1\le j_\mu\le l_{n-1}$ with $\mu\in K^{(n-1)}_{j_\mu}$.

By Lemma \ref{usc} (noting that $\ex K\subset\mtxe$) and since $I_\mu^0$ is dense in $\C$, there exist clopen subsets
$K^{(n)}_\mu\subset K^{(n-1)}_{j_\mu}$ and $W^{(n)}_\mu\subset W^{(n-1)}_{j_\mu}$
with $\mu\in K^{(n)}_\mu$ and
$$
\max\{\diam(K^{(n)}_\mu),\diam(W^{(n)}_\mu)\}<2^{1-n},
$$
satisfying
$$
\psi^{(k)}(K^{(n)}_\mu\times{W^{(n)}_\mu})\subset[0,2^{1-n}],
$$
for all $0\le k\le n$.

By compactness, $\ex K$ can be covered by finitely many sets $K^{(n)}_{\mu_{n,1}},\dots,K^{(n)}_{\mu_{n,l_n}}$. Since these sets are clopen, by subsequent subtracting we can make them disjoint (and still covering $\ex K$) and denote as $K^{(n)}_1,\dots,K^{(n)}_{l_n}$. Correspondingly, we also enumerate $W^{(n)}_{\mu_{n,1}},\dots,W^{(n)}_{\mu_{n,l_n}}$ as $W^{(n)}_1,\dots,W^{(n)}_{l_n}$. Clearly, the properties (1), (2) and (3) are now satisfied for $n$.

\medskip
Once the induction is completed, we continue as follows. Given $n\in\na$, let $s_n:\ex K\to\C$ be a simple function assuming on each set $K^{(n)}_i$ a constant value $t_i^{(n)}\in{W^{(n)}_i}$ ($1\le i\le l_n$). Since the sets $K^{(n)}_i$ are clopen, $s_n$ is continuous. Finally, let $s:\ex K\to\C$ be the limit function of $(s_n)_{n\in\na}$. By (1) and (2) $s$ is a uniform limit of a \sq\ of continuous functions, hence $s$ is continuous as well. Also, the conditions (2) and (3) imply that $\psi^{(k)}(\mu,s(\mu))\le2^{-n}$, for each $\mu\in\ex K$, $n\in\na$ and every $k\le n$. This yields $\psi^{(k)}(\mu,s(\mu))=0$, for all $\mu\in\ex K$ and every $k\in\na$. Thus, for any $\mu\in\ex K$ we have $s(\mu)\in I_\mu^0$, which ends the proof of the lemma.
\end{proof}

\medskip
We return to the proof of Theorem~\ref{submain}. Recall that we aim to constructing a \zd\ system \xtd\ and a continuous affine injection $\pi'':K\to\mtxd$ such that for any $\mu\in K$ and $\mu'' = \pi''(\mu)$, the systems \xmt\ and \xmtd\ are measure-theoretically isomorphic. The mapping $\pi''$ is not going to be surjective and its image is going to be a face $K''$ of $\mtxd$.
\medskip

For a fixed $k\in\na$ and any $\eta\in\{0,1\}^{\{1,\dots,m_k\}}$ we let
$$
P^\eta=\bigcup_{t\in\C}(P_t^\eta\times\{t\}).
$$
Now, the collection
$$
\p^{(k)}=\{P^\eta:\eta\in\{0,1\}^{\{1,\dots,m_k\}}\}
$$
is a finite measurable partition of $X\times\C$ labeled by the elements $\eta$. Let $N(x,t)$ be the array-name of the point $(x,t)\in X\times\C$ under the action $T\times\id$ with respect to the \sq\ of partitions $(\p^{(k)})_{k\in\na}$ of $X\times\C$, which uses the labels $\eta\in\{0,1\}^{\{1,\dots,m_k\}}$ as the alphabet in the $k$th row of an array (notice that if $P_t^\eta$ is empty, the symbol $\eta$ will not appear in $N(x,t)$ for any $x\in X$). Set
$$
X''=\overline{\{(N(x,t),t):(x,t)\in X\times\C\}}.
$$
Let $T'':X''\to X''$ be given by $T''=\sigma\times\id$, where $\sigma$ denotes the shift on arrays. By a standard argument, the \ds\ \xtd\ is a \zd\ extension of $X\times\C$ via a factor mapping $\pi_{0}:X''\to X\times\C$ which is 1-1 (has singleton fibers) except for points $(x,t)$ whose orbits visit the boundary of some member of $\p^{(k)}$. The mapping $\pi_0$ yields a continuous mapping $\pi_0^*$ from $\mtxd$ onto $\mathcal M_{T\times\id}(X\times\C)$. So, by a simple argument (see \cite[Lemma 4.1]{DSurvey}), the inverse mapping defined on the set of measures which have a unique preimage, is a homeomorphism in the relative topologies.

We note here that the section at a level $t$ of $\partial\p^{(k)}$ is contained in $\partial\U^{(k)}_t=\partial\p^{(k)}_t$. In fact, if $x\in X\setminus\partial\p^{(k)}_t$ then
$x$ belongs to some $P^\eta_t$ together with some $\delta$-ball around $x$, which implies that
for all $t'$ sufficiently close to $t$, the $(\delta/2)$-ball around $x$ is contained in
$P^\eta_{t'}$. This implies that $(x,t)$ is not in $\partial P^\eta$, i.e., $x$ is not in the $t$-section of this boundary.

Let $s$ be the continuous selector function defined in the statement of Lemma~\ref{selector}. For $\mu\in\ex K$ we have $\mu\times\delta_{s(\mu)}\in\mathcal{M}^e_{T\times\id}(X\times\C)$, where $\delta_t$ denotes the Dirac measure at $t$. Note that
$$
\mu\times\delta_{s(\mu)}(\bigcup_{k\in\na}\partial\p^{(k)})\le
\mu(\bigcup_{k\in\na}\partial\p^{(k)}_{s(\mu)})=0.
$$
Therefore the mapping $\pi_0$ is 1-1 on a set of full measure $\mu\times\delta_{s(\mu)}$, which implies that $\mu\times\delta_{s(\mu)}$ has a unique preimage by $\pi_0^*$ (which we denote by $\mu''$) and, moreover, the system \xmtd\ is measure-theoretically isomorphic to $(X\times\C,\mu\times\delta_{s(\mu)},T\times\id)$ (and hence, trivially, via the projection onto the first coordinate, to \xmt). Clearly, the mapping $(\mu,t)\mapsto\mu\times\delta_t$ is continuous from $K\times\C$ to $\mathcal{M}_{T\times\id}(X\times\C)$. Further, the graph of $s$ is a compact subset of the domain of this mapping and maps bijectively onto the set $\{\mu\times\delta_{s(\mu)}:\mu\in\ex K\}$, hence the latter set is also compact. Since, as we have observed earlier, the points in the latter set have singleton preimages by $\pi_0$, the inverse of $\pi_0^*$ is a homeomorphism between the set $\{\mu\times\delta_{s(\mu)}:\mu\in\ex K\}$ and its preimage by $\pi^*_0$ (i.e., the set $\{\mu'':\mu\in\ex K\}$), which is hence compact as well. The composition
$$
\mu''\mapsto\mu\times\delta_{s(\mu)}\mapsto(\mu,s(\mu))\mapsto\mu
$$
serves as a homeomorphism between $\{\mu'':\mu\in\ex K\}$ and $\ex K$, moreover, the measures corresponding to each other by this mapping are isomorphic.

Using the ergodic decomposition, the above mapping can be prolonged\footnote{We use the term ``prolongation'' (of a function) for what is customary described as ``extension'', because the word ``extension'' has in this note a defined meaning (opposite to ``factor'').} to an affine and continuous mapping from the compact convex hull spanned by the set $\{\mu'':\mu\in\ex K\}$ onto $K$. Since the measures $\mu''$ are ergodic (being isomorphic to the corresponding ergodic measures $\mu$), this compact convex hull is a face of $\mtxd$. Since both simplices are Bauer, the prolongation is injective (see e.g. \cite[Appendix A.2.5]{DEntropy}) and hence a homeomorphism (which is obviously affine), and maintains the property that the measures corresponding to each other by this mapping are isomorphic. Now, we can define $\pi''$ as the inverse of the above prolonged mapping.
\end{proof}

\medskip
\begin{proof}[Proof of Theorem \ref{main}]
Theorem \ref{main} now becomes a direct con\sq\ of Theorem~\ref{submain} and \cite[Theorem 4.1]{DFace} applied to the system \xtd\ and the compact convex hull spanned by the set $\{\mu'':\mu\in K\}$, which is a face $K''$ in $\mtxd$. Each measure $\mu''$ is isomorphic to some \im\  $\mu\in K$, where $K$ is assumed to not contain periodic measures. So $\mu''$ is aperiodic. Thus the face $K''$ contains no periodic measures, as required in the cited theorem.
\end{proof}

\medskip
\section{Proof of the main result}
The proof of Theorem \ref{sigma} relies on the following three lemmas.
\begin{lem}\label{dense}
The simplex of \im s $\mtx$ of any aperiodic \zd\ system \xt\ ``appears densely'' in the universal simplex, namely, inside any open subset of the universal simplex, we can find a face $\widetilde K$ such that the natural assignment of the universal system restricted to $\widetilde K$ is a copy of $\mtx$ (we will briefly say that $\widetilde K$ is a copy of $\mtx$).
\end{lem}
\begin{proof}
Take any aperiodic \zd\ system $(X,\sigma)$. Since $(X,\sigma)$ is conjugate to a subsystem of the universal system, we can assume that $X$ is an invariant closed subset of $\C^\z$ and $\sigma$ denotes the shift. Let $(k_i,n_i)$ be a \sq\ of integer-valued vectors such that both coordinates increase to infinity as $i$ grows, and let $\R_i$ denote the collection of all $(k_i\times n_i)$-rectangles. Note that all open sets in the universal simplex of the form
$$
U(\mu,i_0,\epsilon)=\{\nu:|\mu(R)-\nu(R)|<\epsilon\text{ for all }R\in\R_i\text{ and }i\le i_0\}
$$
with $\mu$ ranging over ergodic measures, $i_0\in\na$, and $\epsilon>0$,
form a base of the weak-star topology. Fix a set $U(\mu,i_0,\epsilon)$ in the universal simplex. There exists a point $x_0\in\C^\z$ generic for $\mu$, and a constant $N_0>\max\{2n_{i_0}/\epsilon,k_{i_0}\}$, such that for each $N\ge N_0$ the frequency of occurrences of any $R\in\R_i$ with $i\le i_0$ in the rectangle $[1,N_0]\times[0,N]$ of $x_0$ equals $\mu(R)$ up to an error of $\epsilon/2$. In $X$ there exists a clopen $N_0$-marker $F$ visited by orbit of each $x\in X$ with gaps ranging between $N_0$ and $2N_0-1$. For each $x\in X$, we define a new array $\widetilde x$ with rows enumerated from $-N_0+1$ to $+\infty$ as follows: For $i>0$ and any $n\in\z$, $\widetilde x(i,n)=x(i,n)$. The contents of the rows with indices in $[-N_0+1,0]$ is described below: Let $n$ and $n+N$ be two consecutive times of visits of the orbit of $x$ in $F$. We define $\widetilde x$ on the rectangle $[-N_0+1,0]\times[n,n+N-1]$ by
$$
\widetilde x(i-N_0,n+k)=x_0(i,k),\ \ \ i\in[1,N_0], k\in[0,N-1].
$$
Let $\phi(x)$ denote the array $\widetilde x$ with the enumeration of rows shifted so that the rows of $\phi(x)$ are indexed from $1$ to $+\infty$. From the construction of $\phi$ we see that $\phi:X\to\phi(X)\subset\C^\z$ is continuous and one-to-one, and thus, $\phi(X)$ is \tl ly conjugate to $X$, so $\mathcal{M}_\sigma(\phi(X))$ is a copy of $\mathcal{M}_\sigma(X)$. By the choice of $N_0$, it is not hard to estimate that for any $x\in X$, any rectangle $R\in\R_i$ with $i\le i_0$ occurs in $\phi(x)$ with frequency equal to $\mu(R)$ up to an error of $\epsilon$. This proves that $\mathcal{M}_\sigma(\phi(X))$ is a subset of $U(\mu,n_0,\epsilon)$.
\end{proof}

\begin{lem}\label{assume}
Suppose that $E$ is a $\sigma$-compact \zd\ set represented as a countable union $\bigcup_{n\ge 1}E_n$, where every set $E_n$ is compact. Then $E=\bigcup_{n\ge 1}\widetilde E_n$, where the sets $\widetilde E_n$ are compact, pairwise disjoint, and each $\widetilde E_n$ is contained in some $E_m$.
\end{lem}
\begin{proof}
We let $\widetilde E_{0,0}=E_1$. Note that $E_2\setminus E_1$ is relatively open in $E_2$, and thus, can be represented as a countable disjoint union $E_2\setminus E_1=\bigcup_{i\ge1}\widetilde E_{2,i}$, where every set $\widetilde E_{2,i}$ is clopen in $E_2$, and thus, is closed in $E$. We proceed analogously countably many times, namely, for each $n\ge2$, the difference $E_n\setminus(E_1\cup\dots\cup E_{n-1})$ is relatively open in $E_n$, and thus, can be represented as a countable disjoint union $\bigcup_{i\ge1}\widetilde E_{n,i}$, where every $\widetilde E_{n,i}$ is clopen in $E_n$, and thus, is closed in $E$. It now suffices to rearrange the double \sq\ $(\widetilde E_{n,i})_{n,i}$ into a single \sq\ $(\widetilde E_n)_n$.
\end{proof}

\begin{lem}\label{Lusky}
Let $K$ be a simplex and let $F$ be a face of $K$ which is $\epsilon$-dense in $K$
(i.e., for each $x\in K$ there is some $y\in F$ with $\dist(x,y)<\epsilon$).
Then there exists an affine continuous retraction (i.e., a map which is identity on its range), 
$\theta:K\to F$ such that $\dist(\theta(x),x)\le\epsilon$, for all $x\in K$.
\end{lem}

\begin{proof}
For each $x\in K\setminus F$ let $\Theta(x)=\{y\in F: \dist(x,y)<\epsilon\}$ and $\bar\Theta(x)=\overline{\Theta(x)}$. For $x\in F$ define $\Theta(x)=\bar\Theta(x)=\{x\}$. By $\epsilon$-density and convexity of $F$, and convexity of the metric, the multifunction $\Theta$ has nonempty convex images and $\bar\Theta$ has nonempty compact and convex images. Moreover, $\bar\Theta$ is lower hemicontinuous, i.e., for every relatively open set $U\subset F$ the ``preimage'' $\{x\in K:\bar\Theta(x)\cap U\neq\emptyset\}$ is open. Indeed, we have
$$
\{x\in K:\bar\Theta(x)\cap U \neq\emptyset\}= \{x:\Theta(x)\cap U \neq\emptyset\}=\{x:\dist(x,U)<\epsilon\},
$$
which is open (regardless of the \tl\ properties of $U$, which in this case is neither open nor closed in $K$). Also note that convexity of the metric and the fact that $F$ is a face (hence no point in $F$ is a convex combination involving points from outside $F$) imply that $\Theta$ is \emph{convex}, i.e., satisfies,
for any $x,y\in K$ and $\alpha\in(0,1)$ the inclusion
$$
\alpha\Theta(x)+(1-\alpha)\Theta(y) \subset \Theta(\alpha x +(1-\alpha)y).
$$
Clearly, $\bar\Theta$ is convex as well.

At this point we can apply the well-known Lazar-Michael selection theorem (see \cite{La60, FLP01})
which asserts that $\bar\Theta$ admits an affine and continuous selector (i.e., a function $\theta:K\to F$ such that $\theta(x)\in\bar\Theta(x)$ for all $x\in K$). Clearly, $\theta$ satisfies the assertion of the lemma.
\end{proof}

\begin{proof}[Proof of Theorem \ref{sigma}]
As in the proof of Theorem \ref{main}, most of the effort will be devoted to \emph{embedding} $\Phi$ on $K$ in a \zd\ system (which is the same as embedding it in the universal simplex). Surjectivity will be taken care of in the last paragraph of the proof. 
\medskip

Using the fact that the universal simplex is Poulsen, and that the Poulsen simplex is universal in the sense that every simplex is affinely homeomorphic to a face of the Poulsen simplex (see \cite{LOS78}), we can assume that $K$ is a face of the universal simplex. From now on, for any face $F$ of the universal simplex, we will abbreviate ``the restriction to $F$ of the natural assignment coming from the universal simplex'' shortly as the ``restricted natural assignment on $F$''. So, on $K$ we have two assignments: $\Phi$ and the restricted natural assignment (which, at this stage, is beyond our control). 

Recall that $\ex K = \bigcup_{n\ge 1}E_n$, where each $E_n$ is zero-dimensional and compact and the restriction of $\Phi$ to the simplex $K_n$ spanned by $E_n$ is embeddable in a \ds. Lemma \ref{assume} provides a partition of $\ex K$ into disjoint compact sets $\widetilde E_n$, each contained in some $E_m$.  Then $\Phi$ restricted to each $\widetilde E_n$ remains embeddable in a \ds. In other words, we can assume from the start that the sets $E_n$ are disjoint. We choose a summable \sq\ $(\epsilon_k)_{k\ge 1}$ of positive numbers and continue by induction, as follows:
\bigskip

STEP 1. \ \ For a sufficiently large $n_1$, the (disjoint) union $E_1\cup E_2\cup\cdots\cup E_{n_1}$ is $\epsilon_1$-dense in $\ex K$. Then the simplex $L_1$ spanned by this union is $\epsilon_1$-dense in $K$. 
Each of the sets $E_n$ ($1\le n\le n_1$) can be partitioned into a finite union of disjoint clopen sets of diameters smaller than $\epsilon_1$. We denote the finitely many sets obtained in this manner from all the sets $E_1,\dots,E_{n_1}$ by $E_{1,1},\dots,E_{1,m_1}$, and we let $K_{1,1},\dots,K_{1,m_1}$ be the Bauer simplices spanned by $E_{1,1},\dots, E_{1,m_1}$, respectively. In a simplex, faces with disjoint sets of extreme points are disjoint. This implies that the faces $K_{1,i}$ are disjoint. All of these simplices have diameters smaller than $\epsilon_1$ and their union spans $L_1$. The restriction of $\Phi$ to each $K_{1,i}$ $(i=1,\dots,m_1)$ is embeddable in a \ds, and by Theorem~\ref{main} it is realizable in a \zd\ system. 

Now we apply the affine continuous retraction of Lemma \ref{Lusky}, which we denote by $\Ret_1:K\to L_1$, and which moves points by less than $\epsilon_1$. 

Next, using Lemma \ref{dense}, for each $i=1,\dots,m_1$, in the $\epsilon_1$-neighborhood of the set $K_{1,i}$ (within the universal simplex) we find a face $K'_{1,i}$ affinely homeomorphic to $K_{1,i}$ and such that the restricted natural assignment on $K'_{1,i}$ is equivalent to $\Phi|_{K_{1,i}}$. Moreover, since every proper face of any simplex is nowhere dense in that simplex\footnote{Indeed, let $x\in F$, where $F$ is a proper face in a simplex $K$ and let $e\in\ex K\setminus\ex F$. Then the \sq\ $\frac1n e+\frac{n-1}nx$ tends to $x$ from outside $F$.}, we can easily arrange that the faces $K'_{1,i}$ are pairwise disjoint and disjoint from $K$. For each $i$ we choose an affine homeomorphism from $K_{1,i}$ onto $K'_{1,i}$ which establishes the equivalence between $\Phi|_{K_{1,i}}$ and the restricted natural assignment on $K'_{1,i}$. The union of these maps is then prolonged harmonically to an affine homeomorphism $\pi_1$ from $L_1$ onto the simplex $L'_1$ spanned by the union $K'_{1,1}\cup\cdots\cup K'_{1,m_1}$. Notice that $\pi_1$ does not move points more than $2\epsilon_1$ (this is obvious for points in each $K_{1,i}$, then the property passes to the harmonic prolongation by convexity of the metric). The composition $\phi_1=\pi_1\circ\Ret_1$ maps $K$ onto $L'_1$ moving points by less than $3\epsilon_1$, and on $L_1$ it coincides with $\pi_1$ and establishes an equivalence between $\Phi|_{L_1}$ and the restricted natural assignment on $L'_1$.

\medskip
STEP $k\!+\!1$. \ \ Suppose that for some $k\ge 1$ we have constructed an affine continuous map $\phi_k:K\to L'_k$ onto some face of the universal simplex disjoint from $K$, so that on the face $L_k$ of $K$ spanned by the union $E_1\cup\dots\cup E_{n_k}$ it coincides with an affine homeomorphism $\pi_k:L_k\to L'_k$ which establishes an equivalence between $\Phi|_{L_k}$ and the restricted natural assignment on $L'_k$. We start by applying the convex combination 
$$
\phi'_k=(1-\alpha_k)\phi_k+\alpha_k\id,
$$ 
where $\alpha_k$ is positive, but small enough so that at each point of $K$, $\phi'_k$ differs from $\phi_k$ (in the distance) by less than $\epsilon_{k+1}$. Contrary to the noninjective map $\phi_k$, $\phi'_k$ is easily seen (using disjointness of $K$ and $L'_k$) to be an affine homeomorphism between $K$ and its image (however, the good assignment on the image of $L_k$ is now lost, moreover, the image, although remains a simplex, is no longer a face of the universal simplex). Let $L_{k+1}$ be a face of $K$ spanned by a disjoint union of $L_k$ and sufficiently many sets $E_{n_k+1},\dots,E_{n_{k+1}}$, so that $\phi'_k(L_{k+1})$ is $\epsilon_{k+1}$-dense in $\phi'_k(K)$. As before, we partition the sets $E_n$ ($n_k<n\le n_{k+1}$) into smaller clopen sets denoted by $E_{k+1,1},\dots,E_{k+1,m_{k+1}}$, such that, for each $i=1,\dots,m_{k+1}$, $\phi'_k(E_{k+1,i})$ has diameter smaller than $\epsilon_{k+1}$. Clearly, $\Phi$ restricted to each $E_{k+1,i}$ is embeddable in a \tl\ system and thus, by Theorem \ref{main}, realizable in a \zd\ system. 

We can now apply the retraction 
$$
\Ret_{k+1}:\phi_k'(K)\to\phi_k'(L_{k+1}),
$$
which moves points by less than $\epsilon_{k+1}$ and leaves the points of $\phi'_k(L_{k+1})$ invariant. 

\begin{figure}[ht]
\includegraphics[width=14cm]{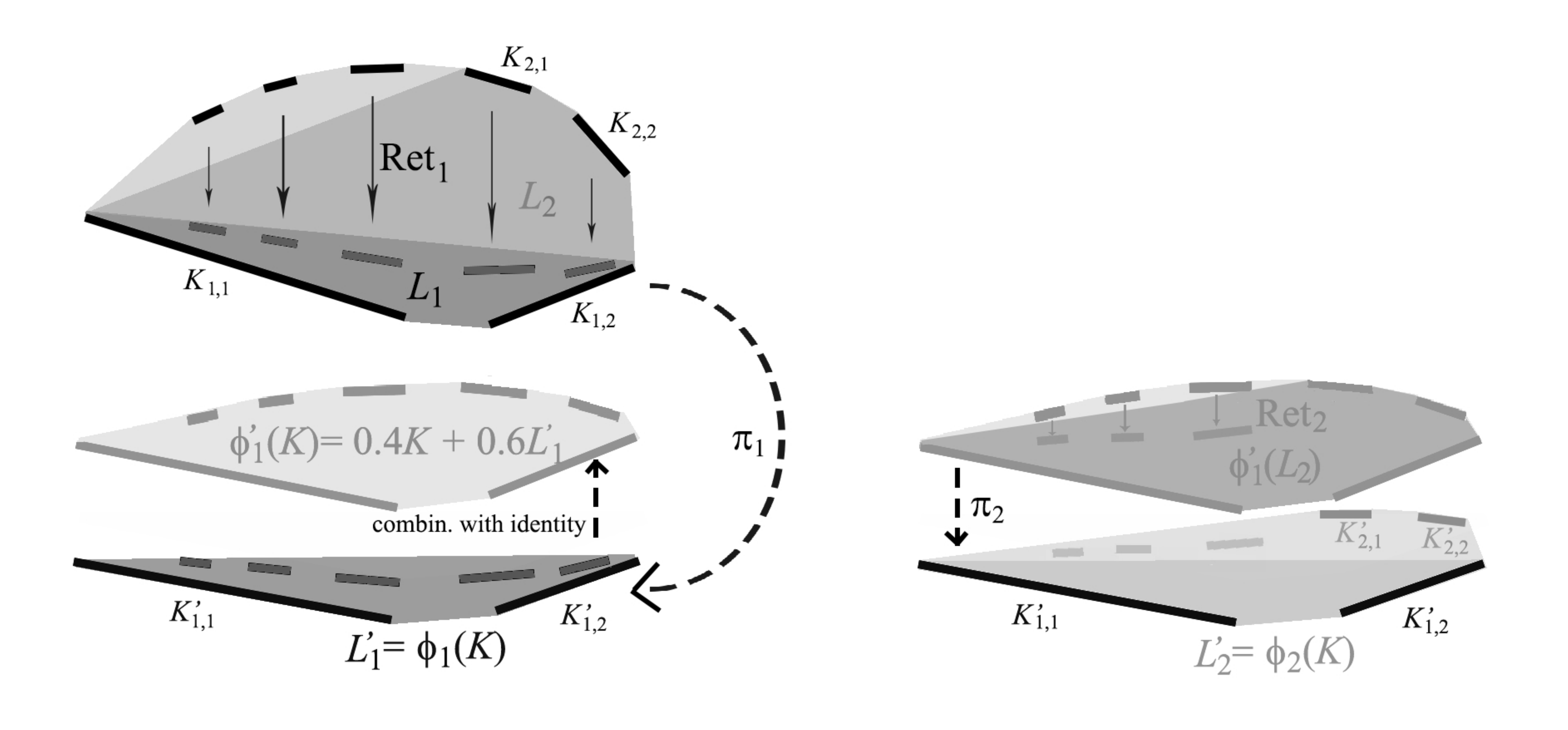}
\end{figure}

In the $\epsilon_{k+1}$-neighborhood of each set $\phi'_k(K_{k+1,i})$ we find a face (of the universal simplex) $K'_{k+1,i}$ affinely homeomorphic to $K_{k+1,i}$, on which the restricted natural assignment is equivalent to $\Phi|_{K_{k+1,i}}$. Moreover, we can arrange that these faces are disjoint from each other, from $L'_k$ and from $K$. We let $L'_{k+1}$ be the simplex spanned by the union $L'_k\cup K'_{k+1,1}\cup\cdots\cup K'_{k+1,m_{k+1}}$. This simplex is affinely homeomorphic to $L_{k+1}$ (and thus also to $\phi'_k(L_{k+1})$) and we can select an affine homeomorphism $\pi_{k+1}:\phi'_k(L_{k+1})\to L'_{k+1}$ so that:
\begin{itemize}
	\item on $\phi'_k(L_k)$ it acts by the formula $\pi_{k+1}(\phi'_k(x)) = \phi_k(x)$.
	Then it sends $\phi'_k(L_k)$ onto $L'_k$ ``forgetting'' the effect of the convex combination with the 
	identity and it establishes an equivalence between $(\phi'_k\circ\Phi)|_{\phi'_k(L_k)}$ (the transported 
	by $\phi'_k$ assignment $\Phi|_{L_k}$) and the restricted natural assignment on $L'_k$,
	\item on each set $\phi'_k(K_{k+1,i})$ it establishes an equivalence between 
	$(\phi'_k\circ\Phi)|_{\phi'_k(K_{k+1,i})}$ (the transported by $\phi'_k$ assignment $\Phi|_{K_{k+1,i}}$) 
	and the restricted natural assignment on $K'_{k+1,i}$; notice that on each set $\phi'_k(K_{k+1,i})$,   
	$\pi_{k+1}$ moves points by less than $2\epsilon_{k+1}$.
\end{itemize}
Then we define $\phi_{k+1}$ as the composition $\pi_{k+1}\circ\Ret_{k+1}\circ\phi'_k$ (see the figure above). This map $\phi_{k+1}$ has the following properties:
\begin{enumerate}
	\item it is an affine continuous map from $K$ onto $L'_{k+1}$ which on $L_k$ coincides with $\phi_k$,
	\item on $L_{k+1}$ it establishes an equivalence between $\Phi|_{L_{k+1}}$ and 
	the restricted natural assignment on $L'_{k+1}$,
	\item it differs from $\phi'_k$ by less than $3\epsilon_{k+1}$, and thus from $\phi_k$ by less than 
	$4\epsilon_{k+1}$.
\end{enumerate}

\medskip
When the induction is complete, we define $\phi$ as the pointwise limit of the maps $\phi_k$. By (3) and summability of the \sq\ $(\epsilon_k)_{k\ge 1}$, the limit exists and is uniform, so $\phi$ is continuous (and it is clearly affine). On each set $L_k$ the limit is achieved in step $k$ (further mappings coincide with $\phi_k$), hence $\phi(L_k)=L'_k$, and $\phi$ establishes an equivalence between $\Phi|_{L_k}$ and the restricted natural assignment on $L'_k$. In particular, $\phi$ is injective on $L_k$.
Thus $\phi$ is injective on $\ex K=\bigcup_{k\ge 1}\ex L_k$, and it sends extreme points of $K$ to points extreme in the universal simplex, and thus extreme in the image $K'$. This implies that $\phi$ is injective on the entire simplex $K$ and hence is an affine homeomorphism between $K$ and $K'$. Also, it establishes an equivalence between $\Phi|_{\ex K}$ and the restricted natural assignment on $\ex K'$, which, by harmonicity of both assignments, implies that $\phi$ establishes an equivalence between $\Phi$ and the restricted natural assignment on $K'$. This ends the proof of embeddability of $\Phi$ in the universal simplex, i.e., in a \zd\ system.

\medskip
To complete the proof we apply once again \cite[Theorem 4.1]{DFace}, this time to the universal simplex and its face $K'$. Each ergodic measure $\mu'\in \ex K'$ is isomorphic to some ergodic measure in the assignment $\Phi$ and hence is aperiodic. Thus, according to the above cited theorem, the restricted natural assignment on $K'$ can be realized in a \zd\ system. Since we have just constructed an equivalence between $\Phi$ and the restricted natural assignment on $K'$, the assignment $\Phi$ can be realized in a \zd\ system as well.
\end{proof}

\section{Final remarks}
{\bf 1.} Both main Theorems \ref{main} and \ref{sigma} hold for noninvertible \tl\ \ds s \xt\ (where $T:X\to X$ is a continuous map, not necessarily a homeomorphism). One has to extend the notion of an assignment admitting as values measure-preserving endomorphisms (rather than automorphisms). The proof of Theorem \ref{main} does not depend on invertibility (also \cite[Theorem 4.1]{DFace}, used in the last stage, holds for endomorphisms). The Krieger's Marker Lemma (Theorem~\ref{krieger}) also remains valid for noninverible systems (see e.g. \cite[Lemma 1]{D06}). The only delicate place is in the proof of Lemma \ref{dense}, where we define the additional rows of $\widetilde x$ by placing there certain rectangles \emph{to the right} of the $N_0$-markers. In the noninvertible case, all arrays have only nonnegative column numbers, and the first marker usually occurs at some positive position. Then we have no indication as to what should be placed in the additional rows \emph{to the left} of the first marker. To cope with this problem we first fill the new rows only to the right of the first $N_0$-marker and then we shift the contents of these rows to the left by $2N_0-1$ units. In this manner the unfilled left section will certainly disappear. 
\medskip

{\bf 2.} It seems that the proof of Theorem \ref{sigma} could be used to show the following: 
\begin{enumerate}
\item[(A)] Suppose that a simplex $K$ equals the closed convex hull of a countable family of its faces: $K=\overline{\mathsf{conv}}\bigl(\bigcup_{n\ge 1} K_n\bigr)$, where the faces $K_n$ are disjoint, and their diameters tend to zero. Let $\Phi$ be an aperiodic assignment on $K$ such that $\Phi|_{K_n}$ is embeddable in a \zd\ system. Then $\Phi$ can be realized in a \zd\ system.
\end{enumerate}
Alas, there is one place where the proof does not pass: in step $k+1$ we need to be able to partition the sets $E_n=\ex K_n$ with $n_k<n\le n_{k+1}$ into finitely many separated pieces, each spanning a face of $K$, and whose images by $\phi'_k$ have small diameters. This may be impossible if the sets $E_n$ are not compact zero-dimensional. We leave (A) as a conjecture.
\medskip

{\bf 3.} Initially the inductive part of the proof of Theorem \ref{sigma} was supposed to be based on the proof of \cite[Corollary 5.2]{DFace}. While working on the details we have discovered a serious gap in that proof. This paper fixes the gap: a correct proof of \cite[Corollary 5.2]{DFace} is obtained by applying  Theorem \ref{sigma} in the case where all sets $E_n$ (and thus $K_n$) are singletons. 
\medskip

{\bf 4.} Anticipating obvious inquires, we confess that we do not have any example of an aperiodic system with a \zd\ sigma-compact (or Bauer) simplex of \im s, for which the realizability in a \zd\ system would not follow by more direct reasons (for example from the SBP). So, one might criticize our results for lack of evident applicability. True. However, we treat this work as an opportunity to develop and practice new tools and methods in handling the difficult problem of realizing assignments in \zd\ systems. As mentioned in the introduction, the ultimate target is to prove (or disprove) such realizability for the natural assignments in all aperiodic \tl\ \ds s.

\medskip

\end{document}